\documentclass[12pt,a4paper]{article}
\usepackage{mathptmx, enumerate, amsmath, amsthm, amscd, amsfonts, amssymb, graphicx, color}
\usepackage[bookmarksnumbered, colorlinks, plainpages]{hyperref}

\newtheorem{thm}{Theorem}[section]
\newtheorem{cor}[thm]{Corollary}
\newtheorem{lem}[thm]{Lemma}

\theoremstyle{definition}

\theoremstyle{remark}
\newtheorem{rem}[thm]{Remark}
\newtheorem{ex}[thm]{Example}
\numberwithin{equation}{section}

\title{ Operator mean inequalities and Kwong functions}
\author{ Nahid Gharakhanlu\footnote{nahid.gh80@gmail.com}~, Mohammad Sal Moslehian\footnote{moslehian@um.ac.ir}~, Hamed Najafi\footnote{hamednajafi20@gmail.com}
\\
\small Department of Pure Mathematics, Ferdowsi University of Mashhad,\\ \small P. O. Box 1159, Mashhad, 91775, Iran.}

\begin{document}
\maketitle
\noindent \abstract{In this paper, we study operator mean inequalities for the weighted arithmetic, geometric and harmonic means. We give a slight modification of Audenaert's result to show the relation between Kwong functions and operator monotone functions. Operator mean inequalities provide some analogs of the geometric concavity property for Kwong functions, operator convex, and operator monotone functions. Moreover, we give our points across by way of some examples which show the usage of our main results.}

\noindent {\textbf{keywords:}} Operator mean inequalities; Kwong functions; Operator convex; Operator monotone.

\section{Introduction}

Let $\mathbb{B}(\mathcal{H})$ denote the algebra of all bounded linear operators on a complex Hilbert space $(\mathcal{H}, \langle \cdot,\cdot\rangle)$. In the case when $\dim \mathcal{H} =n$, we identify $\mathbb{B}(\mathcal{H})$ with the matrix algebra $\mathbb{M}_{n}$ of all $n\times n$ matrices with entries in the complex field $\mathbb{C}$. An operator $A\in \mathbb{B}(\mathcal{H})$ is called positive if $\langle Ax , x\rangle \geq 0$ holds for every $x \in \mathcal{H}$ and then we write $A\geq 0$. An operator $A$ is called strictly positive if $A$ is positive and invertible and we write $A>0$. For self-adjoint operators $A, B\in \mathbb{B}(\mathcal{H})$, we say $A\leq B$ if $B-A\geq 0$. Throughout the paper, we assume that all functions are continuous. Let $f$ be a real-valued function defined on an interval $J$. A function $f$ is called operator monotone if for each self-adjoint operators $A, B\in \mathbb{B}(\mathcal{H})$ with spectra in $J$, $A\leq B$ implies $f(A)\leq f(B)$. A real-valued function $f$ is said to be operator convex if $f((1-\alpha) A+ \alpha B)\leq (1-\alpha) f(A)+ \alpha f(B)$ for all $\alpha \in [0,1]$. $f$ is operator concave if $-f$ is operator convex. For more details, we refer the readers to \cite[Chapter V]{5}.

In addition to the operator monotone functions, Kwong or anti-L\"owner functions have been of interest in recent years. Let $f$ be a  real-valued function on an interval $(0, \infty)$. Kwong or anti-L\"owner matrices associated with $f$ is defined by

\begin{equation*}
K_{f}=\left[ \frac{f(x_{i})+f(x_{j})}{x_{i}+x_{j}}     \right]_{i,j=1}^{n},
\end{equation*}
for any distinct real numbers $x_{1}, ...x_{n}$ in $(0, \infty)$. A function $f$ is  called Kwong function if the Kwong  matrix $K_{f}$ is positive for distinct real numbers $x_{1}, ...x_{n}$ \cite[Chapter 5]{B1}.
 Kwong \cite{14} has shown that if a function $f$ from $(0, \infty)$ into itself is operator monotone, then $f$ is a Kwong  function and so is every non-negative operator monotone decreasing function. The statement for non-negative operator monotone decreasing functions follows easily from this by noting that the function $f$ is Kwong if and only if $\frac{1}{f}$ is  Kwong.  For recent  research papers on Kwong functions and Kwong matrices, we refer the readers to \cite{17, BAK, 15, k2, ZHA} and the references therein. Kwong  functions have applications in the theory of Lyapunov-type equations. Audenaert showed the following relation between Kwong functions and operator monotone functions.
\begin{thm}\emph{\cite[Theorem 2.1]{17}}\label{9}
Let $f$ be a real-valued function on $(0,\infty)$.  Then
\begin{enumerate}[(i)]
\item $f$ is a Kwong function  on $(0,\infty)$ if and only if $\sqrt{t} f(\sqrt{t})$ is a non-negative operator monotone function on $(0,\infty)$.
\item If $f$ is a Kwong function  on $(0,\infty)$, then $\frac{f(\sqrt{t})}{\sqrt{t}}$ is a non-negative operator monotone decreasing function on $(0,\infty)$.
\end{enumerate}
\end{thm}
An operator mean $\sigma$ in the sense of Kubo-Ando \cite{4} is defined by a positive operator
monotone function $f$ on the half interval $(0,\infty)$ with $f(1)=1$ as
\begin{equation*}
A\sigma B=A^{\frac{1}{2}} f\left(A^{-\frac{1}{2}}B A^{-\frac{1}{2}}\right) A^{\frac{1}{2}},
\end{equation*}
for positive invertible operators $A$ and $B$. The function $f$ is called the representing function of $\sigma$. Let $A, B >0$ and $\alpha\in [0,1]$. The $\alpha$-weighted operator geometric mean,  arithmetic mean, and harmonic mean of $A$ and $B$ are defined, respectively, by
\begin{align*}
A\sharp_{\alpha} B &=A^{\frac{1}{2}}\left(A^{-\frac{1}{2}}B A^{-\frac{1}{2}}\right)^{\alpha} A^{\frac{1}{2}},\\
A\nabla_{\alpha} B &= (1-\alpha)A+\alpha B,\\
A!_{\alpha} B &= \left( (1-\alpha)A^{-1}+\alpha B^{-1}\right)^{-1}.
\end{align*}
When $\alpha=\frac{1}{2}$, $A\sharp_{\frac{1}{2}}B$, $A\nabla_{\frac{1}{2}} B$ and $A!_{\frac{1}{2}} B$ are called the operator geometric mean, operator arithmetic mean and operator harmonic mean, and denoted by $A\sharp B$, $A\nabla B$ and $A!B$ respectively. One can easily show that $(A\sharp_{\alpha} B)^{-1}=A^{-1} \sharp_{\alpha} B^{-1}$ for $A, B>0$ and $\alpha\in [0,1]$.

 It is well-known the operator arithmetic-geometric  inquality or Young inequality as 
\begin{equation*}
A\nabla_{\alpha} B \geq A\sharp_{\alpha} B,
\end{equation*}
for $A, B>0$ and $\alpha\in [0,1]$. The study of the reverse Young inequality has received  considerable attention in recent years, and there are several multiplicative and additive reverses of this inequality (see \cite[Chapter 2]{7} and the references therein). Fujii et al. \cite{1} proved that  if $0<mI\leq A, B\leq MI$ with $0< m\leq M$ and $h=\frac{M}{m}$, then
\begin{equation}\label{eq2}
\frac{1}{\sqrt{K(h)}} (A\nabla B) \leq A\sharp B \leq \sqrt{K(h)}  (A ! B),
\end{equation}
where $K(\cdot)$ is the well-known Kantorovich constant  $K(x)=\frac{(1+x)^{2}}{4x}$ for $x>0$.

Tominaga obtained the reverse $\alpha$-weighted arithmetic-geometric mean inequality  as follows.
\begin{thm}\emph{\cite[Theorem 2.1]{Tominaga}}\label{1}
If $0<mI\leq A, B\leq MI$, $\alpha \in [0,1]$, and $h=\frac{M}{m}$, then
\begin{equation*}
A\nabla_{\alpha} B\leq S(h) (A\sharp_{\alpha}B),
\end{equation*}
where the constant Specht's ratio $S(\cdot)$ is defined by $S(x)=\frac{x^{\frac{1}{x-1}}}{e\log x^{\frac{1}{x-1}}}$ for $x> 0$.
\end{thm}
 In this paper, we obtain new sharp inequalities between the weighted arithmetic, geometric and harmonic means. Indeed, In Theorem \ref{5},  we give a new generalization of  \eqref{eq2} and Theorem \ref{1} with a stronger sandwich condition. In Theorem \ref{8}, a slight modification of Theorem \ref{9} is provided. By Corollary \ref{6} and Corollary \ref{3}, we deduce the analogs of the geometric concavity property for Kwong functions and operator convex functions. Moreover, we give some examples illustrating our results.
 
\section{Some arithmetic-geometric mean inequalities}

We start our work with following result.

\begin{thm}\label{5}
Let $A$ and $B$ be positive operators and $\alpha, \beta\in [0,1]$. If $0<sA\leq B\leq tA$ for some scalars $0<s\leq t$, then
\begin{equation*}
\lambda^{-1} (A\nabla_{\alpha} B) \leq  A\sharp_{\beta} B \leq \mu (A!_{\alpha} B),
\end{equation*}
where 
$$
\lambda =\max \left\lbrace  (t^{-\beta} \nabla_{\alpha} t^{1-\beta}),  (s^{-\beta} \nabla_{\alpha} s^{1-\beta})\right\rbrace
$$
and
$$
\mu =\max \left\lbrace  (t^{\beta} \nabla_{\alpha} t^{-(1-\beta)}),  (s^{\beta} \nabla_{\alpha} s^{-(1-\beta)})\right\rbrace .
$$
\end{thm}

\begin{proof}
Without loss of generality assume that $\alpha>0$ and $\beta<1$. The functional calculus shows that to obtain the first desired inequality it is enough to prove
\begin{equation*}
(1-\alpha)+\alpha x \leq \lambda x^{\beta},
\end{equation*}
where $s\leq x\leq t$ and  $\alpha, \beta\in [0,1]$. Put $f_{\alpha, \beta}(x)=\frac{(1-\alpha)+\alpha x}{x^{\beta}}$ for $s\leq x\leq t$. By direct computation, we deduce that
\begin{equation*}
f_{\alpha, \beta}^{'}(x)=\frac{\alpha (1-\beta)}{x^{\beta +1}}\left(x-\frac{\beta (1-\alpha)}{\alpha (1-\beta)}\right).
\end{equation*}
If $s<\frac{\beta(1-\alpha)}{\alpha(1-\beta)}< t$, then $f_{\alpha, \beta}^{'}(x)>0$ when $\frac{\beta(1-\alpha)}{\alpha(1-\beta)}<x<t$, $f_{\alpha, \beta}^{'}(x)<0$ when $s<x<\frac{\beta(1-\alpha)}{\alpha(1-\beta)}$, thus $f_{\alpha, \beta}(x)$ attains it's maximum at $t$ or $s$.\\
If $\frac{\beta(1-\alpha)}{\alpha(1-\beta)}\leq s$, then $f_{\alpha, \beta}^{'}(x)>0$ for $s<x<t$, thus $f_{\alpha, \beta}(x)$ attains it's maximum at $t$.\\
If $t\leq\frac{\beta(1-\alpha)}{\alpha(1-\beta)}$, then $f_{\alpha, \beta}^{'}(x)<0$ for $s<x<t$. Thus $f_{\alpha, \beta}(x)$ attains it's maximum at $s$.\\
From the above discussion, we deduce that $f_{\alpha, \beta}(x)$ attains it's maximum at $t$ or $s$; that is to say
\begin{equation*}
\frac{(1-\alpha)+\alpha x}{x^{\beta}}\leq \max \left\lbrace ( t^{-\beta} \nabla_{\alpha} t^{1-\beta}) , ( s^{-\beta} \nabla_{\alpha} s^{1-\beta})\right\rbrace .
\end{equation*}
Hence, we deduce the following inequality 
\begin{equation*}
(1-\alpha)+\alpha x \leq \lambda x^{\beta},
\end{equation*}
where $s\leq x\leq t$, $\alpha, \beta\in [0,1]$ and $\lambda=\max \left\lbrace ( t^{-\beta} \nabla_{\alpha} t^{1-\beta}) , ( s^{-\beta} \nabla_{\alpha} s^{1-\beta})\right\rbrace$. Therefore,  for every strictly positive operator $X$ with $0< sI\leq X\leq tI $,  we have
\begin{equation*}
(1-\alpha) + \alpha X \leq \lambda X^{\beta}.
\end{equation*}
Since $0<sA\leq B\leq tA$, by substituting $X= A^{\frac{-1}{2}} B A^{\frac{-1}{2}}$ and multiplying $A^{\frac{1}{2}}$ to the both sides in the above inequality, we get 
\begin{equation*}
(1-\alpha) A+\alpha B\leq \lambda (A\sharp_{\beta} B).
\end{equation*}
Thus, we have 
\begin{equation}\label{eq1}
A\nabla_{\alpha} B\leq \max \left\lbrace ( t^{-\beta} \nabla_{\alpha} t^{1-\beta}) , ( s^{-\beta} \nabla_{\alpha} s^{1-\beta})\right\rbrace A\sharp_{\beta} B.
\end{equation}
Since  $\frac{1}{t} A^{-1}\leq B^{-1}\leq \frac{1}{s} A^{-1}$ and $(A^{-1}\nabla_{\alpha} B^{-1})^{-1}= A!_{\alpha} B$, the above inequality deduces that  
\begin{equation}\label{eq3}
A!_{\alpha} B \geq \left( \max \left\lbrace ( t^{\beta} \nabla_{\alpha} t^{-(1-\beta)}) , ( s^{\beta} \nabla_{\alpha} s^{-(1-\beta)})\right\rbrace\right)^{-1} A\sharp_{\beta} B.
\end{equation}
Put $\mu= \max \left\lbrace ( t^{\beta} \nabla_{\alpha} t^{-(1-\beta)}), ( s^{\beta} \nabla_{\alpha} s^{-(1-\beta)})\right\rbrace$. Now, inequalities \eqref{eq1} and \eqref{eq3} gives 
\begin{equation}\label{eq4}
\lambda^{-1} (A\nabla_{\alpha} B) \leq  A\sharp_{\beta} B \leq \mu (A!_{\alpha} B).
\end{equation}
\end{proof}

\begin{rem}
Let $\sigma_{\alpha} $ and  $\tau_{\alpha}$ are two arbitrary means between $\alpha$-weighted arithmetic and harmonic means such that $\nabla_{\alpha}\geq \sigma_{\alpha}\,\, , \,\, \tau_{\alpha}\geq !_{\alpha}$. From the general theory of matrix means and as an immediate consequence of \eqref{eq4}, we have
\begin{equation*}
\lambda^{-1} (A\sigma_{\alpha} B) \leq  A\sharp_{\beta} B \leq \mu (A\tau_{\alpha} B).
\end{equation*}
\end{rem}

\begin{cor}\label{2-1}
Let $A$ and $B$ be positive operators and $\alpha, \beta\in [0,1]$. If $0<mI\leq A, B\leq MI$ for some scalares $0<m\leq M$, then
\begin{equation*}
\lambda^{-1} (A\nabla_{\alpha} B) \leq  A\sharp_{\beta} B \leq \lambda (A!_{\alpha} B),
\end{equation*}
where 
\begin{equation*}
\lambda=\max\left\lbrace (1-\alpha) \left(\frac{m}{M}\right)^{\beta} +\alpha \left(\frac{M}{m}\right)^{1-\beta}, (1-\alpha) \left(\frac{M}{m}\right)^{\beta} +\alpha \left(\frac{m}{M}\right)^{1-\beta} \right\rbrace.
\end{equation*}
\end{cor}
\begin{proof}
Note that  if $0<mI\leq A, B\leq MI$, then $\frac{m}{M} A\leq B\leq \frac{M}{m} A$. Taking $t= \frac{M}{m} $ and $s=\frac{m}{M}$, we have $s=1/t$ and
\begin{align*}
\lambda &=\max \left\lbrace  (t^{-\beta} \nabla_{\alpha} t^{1-\beta}),  (s^{-\beta} \nabla_{\alpha} s^{1-\beta})\right\rbrace ,\\
&=\max \left\lbrace  (s^{\beta} \nabla_{\alpha} s^{-(1-\beta)}),  (t^{\beta} \nabla_{\alpha} t^{-(1-\beta)})\right\rbrace\\
&= \mu
\end{align*}
\end{proof}

\begin{rem}
 Let $x>0$. From \cite[(2.2)]{Tominaga} we infer that
\begin{equation*}
1\leq \frac{(1-\alpha)+\alpha x}{x^{\alpha}}=(1-\alpha) x^{-\alpha}+\alpha x^{1-\alpha}\leq S(x),
\end{equation*}
where $\alpha\in [0,1]$ and  $S(\cdot)$ is the constant Specht's ratio with $S(x)=S\left( \frac{1}{x}\right)$. If $0<m< M$, $0<mI\leq A, B\leq MI$, and $\alpha=\beta$, then 
\begin{equation*}
\lambda =\mu=\max\left\lbrace (1-\alpha) \left(\frac{m}{M}\right)^{\alpha} +\alpha \left(\frac{M}{m}\right)^{1-\alpha}, (1-\alpha) \left(\frac{M}{m}\right)^{\alpha} +\alpha \left(\frac{m}{M}\right)^{1-\alpha} \right\rbrace \leq  S\left( \frac{M}{m}\right).
\end{equation*}
Therefore,
\begin{equation*}
S(h)^{-1} (A\nabla_{\alpha} B)\leq \lambda^{-1}(A\nabla_{\alpha} B) \leq  A\sharp_{\alpha} B \leq  \lambda (A!_{\alpha} B) \leq S(h) (A!_{\alpha} B),
\end{equation*} 
where $h=\frac{M}{m}$. This shows that Corollary \ref{2-1} is a generalization of Theorem \ref{1}.
 \end{rem}
 
 \begin{rem}
Taking $\alpha=\beta$ in Theorem \ref{5}, we have
 \begin{align*}
\lambda &=\max\left\lbrace (1-\alpha) t^{-\alpha} +\alpha t^{1-\alpha}, (1-\alpha) s^{-\alpha} +\alpha s^{1-\alpha} \right\rbrace \leq \max\left\lbrace S(s), S(t) \right\rbrace\\
\mu &=\max\left\lbrace (1-\alpha) t^{\alpha} +\alpha t^{-(1-\alpha)}, (1-\alpha) s^{\alpha} +\alpha s^{-(1-\alpha)} \right\rbrace \leq \max\left\lbrace S\left( \frac{1}{t}\right), S\left( \frac{1}{s}\right) \right\rbrace.
\end{align*}
Put $\gamma=\max\left\lbrace S(s), S(t) \right\rbrace $ to get
\begin{equation*}
\gamma^{-1} (A\nabla_{\alpha} B)\leq \lambda^{-1}(A\nabla_{\alpha} B) \leq  A\sharp_{\alpha} B \leq  \mu (A!_{\alpha} B) \leq \gamma (A!_{\alpha} B).
\end{equation*} 
This shows that Theorem \ref{5} is a generalization of \cite[lemma 1]{ghaemi}.
\end{rem}

\begin{cor}\label{2-2}
Let $A$ and $B$ be positive operators and $\alpha\in [0,1]$. If $0<mI\leq A, B\leq MI$ for some scalares $0<m\leq M$, then
\begin{equation*}
\lambda^{-1} (A\nabla_{\alpha} B) \leq  A\sharp B \leq \lambda (A!_{\alpha} B),
\end{equation*}
where 
\begin{equation*}
\lambda=\left\{
\begin{array}{rl}
 (1-\alpha) \left(\frac{m}{M}\right)^{1/2} +\alpha \left(\frac{M}{m}\right)^{1/2} \,\,&\,\, \text{if }\,\, \alpha\geq \frac{1}{2}\\
 (1-\alpha) \left(\frac{M}{m}\right)^{1/2} +\alpha \left(\frac{m}{M}\right)^{1/2} \,\,&\,\, \text{if }\,\, \alpha\leq \frac{1}{2}.
\end{array}\right.
\end{equation*}
\end{cor}

\begin{rem}
Taking $\alpha=\frac{1}{2}$ in the Corollary \ref{2-2}, we get
\begin{equation*}
\lambda =\frac{1}{2}\left( \sqrt{\frac{M}{m}}+ \sqrt{\frac{m}{M}}\right).
\end{equation*}
If $h=\frac{M}{m}$, then $\lambda =\sqrt{K(h)}$. Thus, Corollary \ref{2-2} gives \eqref{eq2}.
\end{rem}

\section{Kwong functions}
In this section, first we give a slight modification of Theorem \ref{9}. We also show the relation between Kwong functions and operator convex functions.
\begin{thm}\label{8}
Let $f$ be a real-valued function on $(0,\infty)$. 
\begin{enumerate}[(i)]
\item If $p\geq \frac{1}{2}$ and $x^{p} f(x^{p})$ is a non-negative operator monotone function on $(0,\infty)$, then  $f$  is a  Kwong  function on $(0,\infty)$.
\item If $f$  is a  Kwong  function on $(0,\infty)$ and  $0\leq p\leq \frac{1}{2}$, then $x^{p} f(x^{p})$ is a non-negative operator monotone function on $(0,\infty)$.
\item If $f$  is a  Kwong  function on $(0,\infty)$ and  $0\leq p\leq \frac{1}{2}$, then $\frac{f(x^{p})}{x^{p}}$ is a non-negative operator monotone decreasing function on $(0,\infty)$.
\end{enumerate}
\end{thm}

\begin{proof}
Let $f(x)$ be a non-negative operator monotone function on $(0,\infty)$. Then $f(x^p)$  is an operator monotone functions on $(0,\infty)$ for each $0\leq p\leq 1$. Also, if $f$ is a non-negative operator monotone decreasing function on $(0,\infty)$, then $f(x^p)$ is an operator monotone decreasing function for each $0\leq p\leq 1$. Employing these facts, we present the proof of the theorem as follows.
\begin{enumerate}[(i)]
\item Let $x^{p} f(x^{p})$ be operator monotone for $p\geq \frac{1}{2}$. Thus, $x^{p/2p} f(x^{p/2p})=x^{1/2} f(x^{1/2})$ is non-negative operator monotone function on $(0,\infty)$. Now, $f(x)$ is a Kwong function by Theorem \ref{9}. 
\item Suppose $f$ is a Kwong  function on $(0,\infty)$ and $0\leq p\leq \frac{1}{2}$. Then by  Theorem \ref{9}, $x^{1/2} f(x^{1/2})$ is  a non-negative operator monotone function on $(0,\infty)$. So  $x^{2p/2} f(x^{2p/2})=x^{p} f(x^{p})$ is  operator monotone.
\item Let $f$ be a Kwong  function on $(0,\infty)$. It follows from Theorem \ref{9} that $\frac{ f(x^{1/2})}{x^{1/2}}$ is  a non-negative operator monotone decreasing function on $(0,\infty)$. Hence  $\frac{ f(x^{2p/2})}{x^{2p/2}}=\frac{f(x^{p})}{x^{p}}$ is  operator monotone decreasing for $0\leq p\leq \frac{1}{2}$.
\end{enumerate}
\end{proof} 

\begin{thm}\label{2}
Let $g$ be a non-negative operator convex function on $(0,\infty)$ and $-1\leq p\leq 1$. Then $\frac{g(x^{p})}{x^{p}}$ is a Kwong function on $(0,\infty)$.
\end{thm}

\begin{proof}
Note that $\frac{g(x)-g(0)}{x}$ is operator monotone on $(0, \infty)$ \cite[Lemma 2.1]{Uchiyama}. Therefore, $\frac{g(x^p)-g(0)}{x^p}$ is operator monotone increasing or decreasing on $(0, \infty)$ for each $-1\leq p\leq 1$. Hence,
\begin{equation*}
\frac{g(x^p)}{x^p}=\frac{g(x^p)-g(0)}{x^p}+\frac{g(0)}{x^p}
\end{equation*}
is Kwong because the function $x^{p}$ is Kwong for $-1\leq p\leq 1$ \cite[Chapter 5]{B1}.
\end{proof}

Let $f$ be a non-negative continuous function on $(0,\infty)$. The geometric concavity property states that
\begin{equation*}
f(a) \sharp_{\alpha} f(b)\leq f(a \sharp_{\alpha} b) ,
\end{equation*}
for $a,b>0$.
 In what follows, we study some operator arithmetic-geometric mean inequalities involving operator monotone functions. Consequently, we obtain a similar geometric concavity property for Kwong functions and operator convex functions.  First, we recall the following Lemma.

\begin{lem}\emph{\cite[Lemma 4.3.8]{7}}\label{7}
Let $f$ be a non-negative operator monotone function on $(0,\infty)$ and let $A>0$ . Then
\begin{equation*}
f(c A)\leq c f(A)
\end{equation*}
 for every scalar $c\geq 1$.
\end{lem}

Now, we give an operator mean inequality for operator monotone functions.
\begin{thm}\label{3-1}
Let $f$ be a non-negative operator monotone function on $(0,\infty)$. Suppose $A$ and $B$ are positive operators such that $0<sA\leq B\leq tA$ for some scalars $0<s\leq t$. If $\alpha, \beta\in [0,1]$, then 
\begin{equation*}
 f(A)\nabla_{\alpha} f(B) \leq \lambda f(A\sharp_{\beta} B),
\end{equation*}
where 
\begin{align*}
\lambda &=\max \left\lbrace  (t^{-\beta} \nabla_{\alpha} t^{1-\beta}),  (s^{-\beta} \nabla_{\alpha} s^{1-\beta})\right\rbrace.
\end{align*}
\end{thm}

\begin{proof}
 Since $f$ is a non-negative operator monotone function on $(0,\infty)$, it is operator concave \cite[Chapter V]{5}. Therefore,
\begin{equation}\label{eq7}
f(A)\nabla_{\alpha} f(B)\leq f(A \nabla_{\alpha} B).
\end{equation}
It is easy to verify that $t^{-\beta} \nabla_{\alpha} t^{1-\beta}\geq t^{\alpha-\beta}$ and $s^{-\beta} \nabla_{\alpha} s^{1-\beta}\geq s^{\alpha-\beta}$ for $0<s\leq t$ and $\alpha, \beta\in [0,1]$. Without loss of generality, we may assume that $0<s<1<t$. Therefore, $\lambda \geq 1$. Now, we have 
\begin{eqnarray*}
 f(A)\nabla_{\alpha} f(B) &\leq & f(A \nabla_{\alpha} B)\,\,\,\,\,\,\,\, \text{(by \eqref{eq7})}\\
&\leq&  f\left(\lambda (A\sharp_{\beta} B)\right) \,\,\,\,\, \text{(by Theorem \ref{5} and monotonicity of $f$)}\\
&\leq& \lambda f(A\sharp_{\beta} B)\,\,\,\,\, \text{(by Lemma \ref{7})}.
\end{eqnarray*}
\end{proof}

\begin{rem}
We notice that if $\alpha=\beta$ in Theorem \ref{3-1}, then $\lambda\leq \max\left\lbrace S(s), S(t) \right\rbrace$. Therefore,
\begin{equation*}
f(A)\sharp_{\alpha} f(B)\leq f(A)\nabla_{\alpha} f(B) \leq \lambda f(A\sharp_{\beta} B)\leq \gamma f(A\sharp_{\alpha} B),
\end{equation*}
where $\gamma=\max\left\lbrace S(s), S(t) \right\rbrace$. This shows that Theorem \ref{3-1} is a generalization of \cite[Theorem 1]{ghaemi}.
\end{rem}

\begin{cor}\label{3-2}
Let $A$ and $B$ be positive operators such that $0<sA\leq B\leq tA$  for some scalars $0<s\leq t$. If $p, \alpha, \beta\in [0,1]$, then 
\begin{equation*}
 A^p\nabla_{\alpha} B^p \leq \lambda^p (A\sharp_{\beta} B)^p
\end{equation*}
where 
\begin{align*}
\lambda &=\max \left\lbrace  (t^{-\beta} \nabla_{\alpha} t^{1-\beta}),  (s^{-\beta} \nabla_{\alpha} s^{1-\beta})\right\rbrace.
\end{align*}
\end{cor}

\begin{proof}
Note that $f(x)=x^p$ is a non-negative operator monotone function on $(0,\infty)$. Also, in the proof of Theorem \ref{3-1}, we have $ f(A)\nabla_{\alpha} f(B)\leq f\left(\lambda (A\sharp_{\beta} B)\right)$.
\end{proof}

\begin{cor}
Let  $f$ be a non-negative operator monotone function on $(0,\infty)$. Suppose $A$ and $B$ are positive operators such that $0<mI\leq A, B\leq MI$ for some scalares $0<m\leq M$. If $\alpha, \beta\in [0,1]$, then
\begin{equation*}
 f(A)\nabla_{\alpha} f(B) \leq \lambda f(A\sharp_{\beta} B)
\end{equation*}
where 
\begin{equation*}
\lambda=\max\left\lbrace (1-\alpha) \left(\frac{m}{M}\right)^{\beta} +\alpha \left(\frac{M}{m}\right)^{1-\beta}, (1-\alpha) \left(\frac{M}{m}\right)^{\beta} +\alpha \left(\frac{m}{M}\right)^{1-\beta} \right\rbrace.
\end{equation*}
\end{cor}

Let $A, B\in \mathbb{M}_{n}$ be positive matrices. Ando and Hiai \cite{Ando} used the log-majorization to show that if  $\alpha\in[0,1]$ and $p\geq 1$, then 
\begin{equation*}
\vert \vert \vert A^{p}\sharp_{\alpha} B^{p} \vert \vert \vert \leq \vert \vert \vert (A\sharp_{\alpha} B)^{p} \vert \vert \vert ,
\end{equation*}
for every unitarily invariant norm $\vert \vert \vert \cdot \vert \vert \vert $. We recall that a norm $\vert\vert\vert \cdot \vert\vert\vert$ is said to be unitarily invariant if $\vert\vert\vert A \vert\vert\vert = \vert\vert\vert U A V\vert\vert\vert$ for all  matrices $A, U, V$ with $U$ and $V$ unitary. Notice that a unitarily invariant norm is monotone in the sense that $0 < A\leq B$ implies $\vert\vert\vert A \vert\vert\vert \leq \vert\vert\vert B \vert\vert\vert $. If $0<mI\leq A, B\leq MI$, Seo \cite{Seo} obtained an analog of Ando's inequality for $0< p\leq 1$ as follows
\begin{equation}\label{3-2-1}
\vert \vert \vert A^{p}\sharp_{\alpha} B^{p} \vert \vert \vert \leq  (S(h))^{p} \vert \vert \vert (A\sharp_{\alpha} B)^{p} \vert \vert \vert .
\end{equation}
Now, let $A, B\in \mathbb{M}_{n}$ such that $0<sA\leq B\leq tA$ and $0\leq p\leq 1$. Corollary \ref{3-2} shows that
\begin{equation*}
\vert \vert \vert A^{p}\sharp_{\alpha} B^{p}\vert \vert \vert \leq \vert \vert \vert A^{p}\nabla_{\alpha} B^{p}\vert \vert \vert\leq \lambda^{p} \vert \vert \vert (A\sharp_{\beta} B)^{p} \vert \vert \vert,
\end{equation*}
for all $\alpha, \beta\in [0,1]$. Notice that if $0<mI\leq A, B\leq MI$ and $\alpha=\beta$,  then $\lambda \leq S(h)$ and we get Seo's inequality \eqref{3-2-1}.

Next, we deduce the analogs of the geometric concavity property for Kwong and consequently for operator convex functions.

\begin{cor}\label{6}
Let $f$ be a real-valued Kwong function on $(0,\infty)$. Suppose $0\leq p\leq \frac{1}{2}$ and $0<sA\leq B\leq tA$ with $0< s\leq t$. If $\alpha, \beta\in [0,1]$, then 
\begin{equation*}
  (A^{p} f(A^{p})\nabla_{\alpha} B^{p} f(B^{p})) \leq \lambda (A\sharp_{\beta} B)^{p} f((A\sharp_{\beta} B)^{p})
\end{equation*}
where 
\begin{align*}
\lambda &=\max \left\lbrace  (t^{-\beta} \nabla_{\alpha} t^{1-\beta}),  (s^{-\beta} \nabla_{\alpha} s^{1-\beta})\right\rbrace.
\end{align*}
\end{cor}
\begin{proof}
Note that $h(x)=x^{p} f(x^{p})$ is a non-negative operator monotone function on $(0,\infty)$ due to Theorem \ref{8}. Now, it is sufficient to use Theorem \ref{3-1}.
\end{proof}

\begin{ex}
The function $f(x)=\sinh^{-1}x=\ln\left(x+(x^2+1)^{\frac{1}{2}}\right)  $ is Kwong on $(0,\infty)$ \cite[Theorem 2.8]{15}. Applying $f(x)=\sinh^{-1}x$ to Corollary \ref{6} we arrive at
\begin{align*}
  (A^{p}\sinh^{-1}(A^{p})\nabla_{\alpha} B^{p}\sinh^{-1}(B^{p})) \leq \lambda (A\sharp_{\beta} B)^{p}\sinh^{-1}(A\sharp_{\beta} B)^{p}.
\end{align*}
In the case when $\alpha=\beta$, $0<mI\leq A, B\leq MI$ and $h=\frac{M}{m}$, we deduce the following operator geometric mean inequality
\begin{equation*}
(A^{p}\sinh^{-1}(A^{p})\sharp_{\alpha} B^{p}\sinh^{-1}(B^{p}))\leq S(h) (A\sharp_{\alpha} B)^{p}\sinh^{-1}(A\sharp_{\alpha} B)^{p}.
\end{equation*}
\end{ex}

\begin{cor}\label{3}
Let $g$ be a non-negative operator convex function on $(0,\infty)$, $0\leq p\leq \frac{1}{2}$, and $0<sA\leq B\leq tA$ for $0< s\leq t$. If $\alpha, \beta\in [0,1]$, then
\begin{equation*}
 g(A^{p})\nabla_{\alpha} g(B^{p}) \leq \lambda g((A\sharp_{\beta} B)^{p}),
\end{equation*}
where $\lambda$ is as in Corollary \ref{6}.
\end{cor}

\begin{proof}
If  $g$ is a non-negative operator convex function on $(0,\infty)$, then Theorem \ref{2} ensures that $\frac{g(x)}{x}$ is a  Kwong functions on $(0,\infty)$. Thus, applying Corollary \ref{6} to the Kwong function $\frac{g(x)}{x}$, we get the desired inequality.
\end{proof}

\begin{rem}
Corollary \ref{3} gives a similar geometric concavity property for operator convex functions as follows
\begin{equation*}
g(A^{p})\sharp_{\alpha} g(B^{p})\leq \lambda g((A\sharp_{\beta} B)^{p}),
\end{equation*}
for all $\alpha, \beta\in [0,1]$, $0\leq p\leq \frac{1}{2}$ and $0<sA\leq B\leq tA$.
\end{rem}

\bigskip
\textbf{{Acknowledgments.}} N. Gharakhanlu was supported by a grant from the Iran national Elites Foundation (INEF) for a postdoctoral fellowship under the supervision of M. S. Moslehian.
\bigskip
\bibliographystyle{amsplain}

\begin{thebibliography}{99}

\bibitem{Ando} Ando, T; Hiai, F.: \textit{Log majorization and complemetary Golden-Thompson inequalities}, Linear Algebra Appl.  \textbf{197/198} (1994), 113--131.

\bibitem{17}  Audenaert, K.M.R.:  \textit{A characterisation of anti--L\"owner functions}, Proc. Amer. Math. Soc. \textbf{139} (12) (2011), 4217--4223.

\bibitem{BAK} Bakherad, M.: \textit{Some generalized numerical radius inequalities involving Kwong functions}, Hacet. J. Math. Stat. \textbf{48} (2019), no. 4, 951--958. 

\bibitem{5} Bhatia, R.: \textit{Matrix Analysis}, Graduate Texts in Mathematics, 169. Springer-Verlag, New York, 1997.  

\bibitem{B1} Bhatia, R.: \textit{Positive Definite Matrices}, Princeton Series in Applied Mathematics. Princeton University Press, Princeton, NJ, 2007.


\bibitem{1}  Fujii, J.I; Nakamura, M; Pe\v{c}ari\'{c}, J; Seo, Y.:  \textit{Bounds for the ratio and difference between parallel sum and series via Mond-Pe\v{c}ari\'{c} method}, Math. Inequal. Appl.  \textbf{9} (4) (2006), 749--759.

\bibitem{7} Ghaemi, M. B; Gharakhanlu, N; Rassias, T. M; Saadati, R.: \textit{Advances in matrix inequalities}, Springer Optimization and its Applications, 176. Springer, Cham, 2021.  

\bibitem{ghaemi} Ghaemi, M. B; Kaleibary, V.: \textit{Some inequalities involving operator monotone functions and operator means}, Math. Inequal. Appl. \textbf{19}  (2016), 757--764. 


\bibitem{4} Kubo, F;  Ando, T.: \textit{Means of positive linear operators}, Math. Ann.  \textbf{246} (1980), 205--224. 

\bibitem{14} Kwong, M. K.:  \textit{Some results on matrix monotone functions}, Linear Algebra Appl.  \textbf{118} (1989), 129--153.

\bibitem{15} Najafi, H.:  \textit{Some results on Kwong functions and related inequalities}, Linear Algebra Appl.  \textbf{439} (9) (2013), 2634--2641.

\bibitem{k2} Numazawa, Y; Sano, T.:  \textit{Conditional negativity of Kwong matrices II}, Linear Multilinear Algebra.  \textbf{70} (3) (2022), 511--516.
                      
\bibitem{Seo} Seo, Y.:  \textit{Reverses of the Golden-Thompson type inequalities due to Ando-Hiai-Petz}, Banach J. Math. Anal. \textbf{2} (2) (2008), 140--149.

\bibitem{Tominaga} Tominaga, M.: \textit{Specht's ratio in the Young inequality}, Sci. Math. Japon. \textbf{55} (2002), 583--588.


\bibitem{Uchiyama} Uchiyama, M.: \textit{Operator monotone functions, positive definite kernel and majorization}, Proc. Amer. Math. Soc. \textbf{138} (11) (2010), 3985--3996.

\bibitem{Zhan} Zhan, X.: \textit{Matrix inequalities}, Lecture Notes in Mathematics. Springer Berlin, Heidelberg, 2002.  

\bibitem{ZHA} Zhao, J; Wu, J.: \textit{Some matrix inequalities related to Kwong functions}, Linear Multilinear Algebra \textbf{67} (2019), no. 5, 995--1005. 


 
\end{thebibliography}

\end{document}